\documentclass[12pt, a4paper, leqno]{article}
\usepackage{Prem/_preamble}

\begin{document}
\title{Monochromatic $k$ in a row} 
\author{Kuo-Han Ku}
\date{\small Mathematics Division, National Center for Theoretical Sciences, National Taiwan University, Taipei 10617, Taiwan. Email: \href{mailto:kuohan_ku@ncts.ntu.edu.tw}{kuohan\_ku@ncts.ntu.edu.tw}.}

\maketitle
\setcounter{page}{1}
\setcounter{section}{0}
\thispagestyle{fancy}
\pagestyle{fancy}
\lhead{}
\chead{}
\rhead{}
\lfoot{}
\cfoot{\thepage}
\rfoot{}
\renewcommand{\headrulewidth}{0pt}
\renewcommand{\footrulewidth}{0pt}

\begin{abstract}
    We study a variant of the $k$-in-a-row game in which players alternatively claim positions until a $k$-in-a-row is created among all claimed positions. This leads to the constraint \emph{near $k$-in-a-row avoiding} on configurations and the associated problem of determining their extremal densities of such configurations on a given board $B$.

    We investigate this problem on two types of boards: the grid $\mathbb{Z}^2$ and hypercubes $[k]^d$. For the grid $\Z^2$, we establish nearly tight bounds on the maximum density $D(k,\mathbb{Z}^2)$, showing that $D(k,\mathbb{Z}^2)=1-\frac{2}{k}$ whenever $3\nmid k$, and determine both $D(3,\mathbb{Z}^2)$ and $d(3,\mathbb{Z}^2)$ exactly. We also bound the minimum density $d(k,\mathbb{Z}^2)$ up to a gap of $(8+o(1))k^{-1}$. For hypercubes $[k]^d$, we derive asymptotic bounds on $D(k,[k]^d)$ up to order $k^{-2}$ and obtain the exact value of $d(k,[k]^d)$. 
    Our results contrast with the classical no-$(k+1)$-in-line problem, a similar problem imposing different constraint, where the trivial upper bound is conjectured to be attainable.
\end{abstract}

\section{Introduction}

\emph{Five-in-a-row}, also known as ``Gomoku", is a strategy board game played on a $15\times 15$ grid. Two players take turns placing stones of their respective color on the board. The winner is the first to form a consecutive line of five or more stones of their color. Variants of five-in-a-row have been extensively studied, particularly in positional game theory. Using computer-assisted analysis, Allis~\cite{gomoku_bruteforce} showed that standard and free Gomoku are first-player wins. On the other hand, the $9$-in-a-row game played on the infinite grid $\mathbb{Z}^2$ is a draw; a drawing strategy for the second player is given in~\cite{9-in-a-row}.

Now imagine playing Five-in-a-row, but with two rules changed. First, all players place stones of the same color. Second, a player wins by occupying the last empty position of a five-in-a-row. We call this variant \emph{Monochromatic Five-in-a-row}.

Under these new rules, the nature of the game changes completely. Rather than attacking and blocking the opponent, players must avoid creating a five-in-a-row with exactly one empty position, which we call a \emph{near five-in-a-row}. Such a pattern appears only in the last two moves. If one player creates a near five-in-a-row, the other player completes it and wins.

This leads to the fundamental question: how long can such a game last? Clearly, the length of the game is capped by $\max_{S\in\mathcal{S}}|S|+2$, where $\mathcal{S}$ is the family of subsets forbidding near five-in-a-rows. More generally, we can consider the monochromatic $k$-in-a-row played on a given board $B$, where $k\in\N$, and define the family of saturated configurations $\mathcal{S}_k(B):=\max(\{S\subseteq B:|L\backslash S|\geq 2\ \forall \text{ $k$-in-a-row $L$}\})$\footnote{$\max(\mathcal{P})$ denotes the set of maximal elements in a poset $\mathcal{P}$.} for this game. Then the following question comes up.
\begin{question*}
    Given a board $B$ and an integer $k\in\N$, how does the spectrum of densities of the near $k$-in-a-row avoiding family $\mathcal{S}_k(B)$ behave? In particular, what can we say about the extrema of the spectrum?
\end{question*}

A closely related problem in Combinatorial Geometry is the ``No-$(k+1)$-in-line" problem, which asks for the maximum size of a subset of a grid (typically $[n]^2$ or $\mathbb{Z}^2$) containing no $(k+1)$ collinear points. The trivial upper bound for such subsets is $kn$, and it is conjectured that this bound can be attained for all $n\geq k$. In contrast, our results show that, for the problem considered here, the trivial upper bound, $1-\frac{2}{k}$, cannot be achieved in certain setting.

Remarkably, Grebennikov and Kwan~\cite{no-k+1-in-line} recently proved that the bound $kn$ is indeed achievable whenever $n\geq k\geq 10^{37}$. For small values of $k$, the conjecture has been verified in several cases~\cite{ANDERSON1979365, FLAMMENKAMP1992305, FLAMMENKAMP1998108}, while it remains open in general. In the case $k=2$, the problem reduces to the classical no-three-in-line problem. The best known lower bound, $(1.5+o(1))n$, comes from a construction of Hall, Jackson, Sudbery, and Wild~\cite{HALL1975336}. 

In this paper, two types of boards, \emph{the grid} $\Z^2$ and \emph{hypercubes} $[k]^d$, are concerned. Our main results give bounds on the extrema, $D(k,B)$ and $d(k,B)$, of density spectrum for these boards. The two cases require different techniques.

First, we look at the infinite grid $\Z^2$. As we cannot directly define such density on $\Z^2$, in order to approach $\Z^2$, we start from a finite grid board $B=[\ell]^2$ and take $\ell$ to infinity. Denote $D(k,\Z^2):=\lim\limits_{\ell\to\infty} D(k,[\ell]^2)$ and $d(k,\Z^2):=\lim\limits_{\ell\to\infty} d(k,[\ell]^2)$. Later, the existence and uniqueness of these limits will be proved. For $k=3$, both limits are determined.
\begin{theorem}\label{thm1} $D(3,\Z^2)=\frac{1}{5}$ and $d(3,\Z^2)=\frac{1}{17}.$
\end{theorem}
For general $k$, we obtain the following nearly tight bound on the upper limit,
\begin{theorem}\label{thm2}
    For $k\in\N_{\geq 2}$, we have $D(k,\Z^2)\in\left[1-\frac{2}{k-1},1-\frac{2}{k}\right]$. Further, if $k$ is not a multiple of $3$, it achieves upper bound, i.e., $D(k,\Z^2)=1-\frac{2}{k}$.
\end{theorem}
and bounds on the lower limit with a gap $(8+o(1))k^{-1}$.
\begin{theorem}\label{thm3}We have the bounds on the lower limit
    \begin{align*}
        d(k,\Z^2)\in [1-16k^{-1}-o(k^{-1}), 1-8k^{-1}+o(k^{-1})]
    \end{align*}
as $k$ grows.
\end{theorem}

For the second type, the minimum density is known. For the maximum density, we study the extremal behaviour as $k$ goes to infinity while $d$ fixed, and derive asymptotic bounds on the maximum density to the second order of $k^{-1}$.

\begin{theorem}\label{thm4}
For $d\in\N$, we have $D(k,[k]^d)=1-2k^{-1}-O(k^{-2})$ as $k$ grows.
\end{theorem}

\begin{theorem}\label{thm5}
For $k\in\N_{\geq 2}$ and $d\in\N$, we have $d(k,[k]^d)=\left(1-\frac{2}{k}\right)^{d}$.
\end{theorem}

This paper is organized as follows: Section \ref{section2} introduces the necessary definitions and propositions. We discuss the grid in Section \ref{section3}, and hypercubes in Section \ref{section4}. Section \ref{section5} contains conjectures, variants, and possible directions for future research.

\section{Preliminaries}\label{section2}

In this section, we introduce the basic notions used throughout the paper, including $k$-boards, $k$-lines, and near $k$-lines, and describe several common examples. We then define saturated configurations and extremal densities formally. Finally, we prove two important propositions concerning the infinite grid $\mathbb{Z}^2$.

\begin{definition} Let $k\in\N_{\geq 3}$. Given a set $B$ and a $k$-uniform subset family $\mathcal{L}\subseteq \binom{B}{k}$. The pair $(B,\mathcal{L})$ is called a \emph{$k$-board}. A subset $S\in \mathcal{L}$ is called a \emph{$k$-line}. A subset $S'\subseteq B$ is called a \emph{near $k$-line} if $|S'|=k-1$ and $S'\subseteq S$ for some $k$-line $S$.
\end{definition}

In the following we focus on grid-shaped boards that arise from extensions of the classic Five-in-a-row board, in which all $k$-lines are either axis-parallel or diagonal. As long as the set of $k$-lines considered here is totally determined by the ground set, we will simply use the ground set $B$ to represent the corresponding $k$-board $(B,\mathcal{L})$.

\begin{definition}
    Let $k,d\in\N$. The set $\Z^d$ along with the following $k$-uniform subset family
    \begin{align*}
        \left\{p+[0,k-1]\Vec{d}\ :\ p\in\Z^d,\ \Vec{d}\in \{-1,0,1\}^d\backslash\{0\}\right\}
    \end{align*}
    is called the \emph{infinite $k$-board of dimension $d$}. For $\ell=(\ell_1,\ell_2,\cdots,\ell_d)\in \N^d$, consider the hyper-rectangle $B:=\prod_{i=1}^d [\ell_i]$ and the family 
    \begin{align*}
        \mathcal{L}:=\left\{p+[0,k-1]\Vec{d}\ :\ p\in B,\ \Vec{d}\in\{-1,0,1\}^d\backslash\{0\},\ p+[0,k-1]\Vec{d}\subseteq B\right\},
    \end{align*}
    the board $(B,\mathcal{L})$ is called the \emph{(standard) $k$-board of shape $\ell$}. Further, consider the torus $B':=\prod_{i=1}^d (\Z/\ell_i\Z)$ and the family
    \begin{align*}
        \mathcal{L}':=\left\{p+[0,k-1]\Vec{d}\ :\ p\in B',\ \Vec{d}\in\{-1,0,1\}^d\backslash\{0\}\right\},
    \end{align*}
    we call $(B',\mathcal{L}')$ the \emph{toric $k$-board of shape $\ell$}.
\end{definition}

\begin{example} The standard $5$-board of shape $(19,19)$ corresponds to the board of the classical Five-in-a-row game.
\end{example}

\begin{definition}
    Given a $k$-board $(B,\mathcal{L})$. A subset $S\subseteq B$ is called a \emph{configuration} if it contains no near $k$-line. The set of all configurations is denoted by $\mathcal{C}_k(B)$. Further, a configuration $S$ is called \emph{saturated} if it is maximal in $\mathcal{C}_k(B)$ with respect to inclusion. We denote by $\mathcal{S}_k(B)$ the family of all saturated configurations, i.e.,
    \begin{align*}
        \mathcal{S}_k(B):=\max(\{S\subseteq B: |L\backslash S|\geq 2\ \forall L\in \mathcal{L}(B)\}).
    \end{align*}
\end{definition}

\begin{example}
    On the board $B=[k]^d$, take $k':=k-2$, one can find that every line $L$ with $|B\cap L|\geq k'+1$ is either axis-parallel or diagonal. Thus the condition ``No-$(k'+1)$-in-line" agrees with near $k$-line avoiding on this board.
\end{example}

One can verify that every saturated configuration corresponds to a game state right before the final two moves. Thus bounds on size of saturated configurations translate directly into bounds on the possible game lengths, up to an additive constant of $2$. In the viewpoint of density, they are almost the same when the board size goes to infinity. 

\begin{definition}
    Given a finite $k$-board $B$, the set
    \begin{align*}
        \mathcal{D}_k(B):=\left\{\frac{|S|}{|B|}:S\in\mathcal{S}_k(B)\right\}
    \end{align*}
    is called the \emph{density spectrum} of $B$. We also define $D(k,B):=\max\mathcal{D}_k(B)$ and $d(k,B):=\min\mathcal{D}_k(B)$.
\end{definition}

To define the extremal density for infinite grid board, we use the limit of that of finite grid boards.
\begin{definition}
    For $k\in \N$, we define
    \begin{align*}
        D(k,\Z^2):=\lim_{t\to \infty} D(k,[m(t)]\times [n(t)])\ \ \ \text{ and }\ \ \ d(k,\Z^2):=\lim_{t\to \infty} d(k,[m(t)]\times [n(t)]),
    \end{align*}
    where $m,n:\N\to\N$ are two arbitrary maps satisfying $m(t),n(t)\in\omega(1)$.
\end{definition}

The following two propositions provide the existence and well-definedness of these limits.

\begin{proposition}\label{limit_existence}
    Given two maps $m,n:\N\to\N$. If $m(t),n(t)\in \omega(1)$, then the limit of two sequences $\{D(k,[m(t)]\times [n(t)])\}_{t=1}^\infty$ and $\{d(k,[m(t)]\times [n(t)])\}_{t=1}^\infty$ exist.
\end{proposition}

\begin{proposition}\label{limit_uniqueness}
    For $k\in\N$, there are two reals $L_{\text{high}}, L_{\text{low}}\in (0,1)$ such that
    \begin{align*}
        \lim_{t\to\infty} D(k,[m(t)]\times [n(t)])=L_{\text{high}}\ \ \ \text{ and }\ \ \ 
        \lim_{t\to\infty} d(k,[m(t)]\times [n(t)])=L_{\text{low}}
    \end{align*}
    hold for every maps $m,n:\N\to\N$ satisfying $m(t),n(t)\in \omega(1)$.
\end{proposition}

The existence of the limit is based on the asymptotic monotonicity of those sequences. Roughly, we can build a bigger subset via tiling some small configurations, it produces configurations with density $\varepsilon$-close to that of building blocks. This method provides the monotonicity on far-tail terms of the sequence.

\begin{proof}[Proof of Proposition \ref{limit_existence}] Fix a such pair $m(t)$ and $n(t)$, and write 
\begin{align*}
    D(t):=D(k,[m(t)]\times[n(t)])\ \ \text{ and }\ \  d(t):=d(k,[m(t)]\times[n(t)]).
\end{align*}
Since both sequences $\{D(t)\}_{t=1}^\infty$ and $\{d(t)\}_{t=1}^{\infty}$ are bounded, their limit superiors and limit inferiors exist. Define 
\begin{align*}
    L_{\text{high}}:=\limsup_{t\to\infty}D(t)\ \ \ \text{ and }\ \ \ L_{\text{low}}:=\liminf_{t\to\infty}d(t).
\end{align*}
We prove the claim by constructing sequences of configurations that attain these limits. Let $X\subseteq [m_0]\times [n_0]$. For any rectangle $R=[m]\times [n]$, define
\begin{align*}
    S(X,R):=\bigcup_{v\in V} (X+v),
\end{align*}
where $$V:=(m_0+2)\bigl[0,\lfloor\tfrac{m}{m_0+2}\rfloor-1\bigr]\times (n_0+2)\bigl[0,\lfloor\tfrac{n}{n_0+2}\rfloor-1\bigr]\subseteq R.$$ In other words, $S(X,R)$ is obtained by placing disjoint copies of $X$ inside $R$, separated by two empty rows and columns.

Consider any $k$-line $L$ in $R$. If $L$ intersects with two copies of $X$, then it must pass through two empty rows or columns, and hence cannot form a near $k$-line. If $X\in\mathcal{C}_k([m_0]\times [n_0])$, there are no near $k$-lines contained entirely within a single copy of $X$. The only remaining possibility is a diagonal $(k-1)$-line whose endpoints lie on the boundary of $[m_0]\times [n_0]$. Such a line could become a near $k$-line in $R$ if it is fully contained in $X$, we call them \emph{danger lines}. Let $D$ be a set containing one point from each danger line, then we derive a configuration $S(X\backslash D,R)\in\mathcal{C}_k(R)$. 

Let $S'=S'(X,R)\in\mathcal{S}_k(R)$ be a saturated configuration with $S'\supseteq S(X\backslash D,R)$. Since each copy of $X\backslash D$ is saturated in $R$, the additional points in $S'\backslash S$ must be outside the boxes $([m_0]\times [n_0])+v$ for all $v\in V$. By counting the number of such points, we obtain upper and lower bounds on $\frac{|S'|}{|R|}$:
\begin{align*}
    \frac{|S'|}{|R|}\geq \frac{|S|}{|R|}&\geq \frac{(\frac{m}{m_0+2}-1)(\frac{n}{n_0+2}-1)}{mn}(|X|-|D|)\\
    &\geq \frac{m_0n_0}{(m_0+2)(n_0+2)}\frac{|X|}{m_0n_0}-\frac{|X|}{m(n_0+2)}-\frac{|X|}{(m_0+2)n}-\frac{4}{m_0n_0}\\
    &\geq\left(\frac{m_0n_0}{(m_0+2)(n_0+2)}-\frac{2\max\{m_0,n_0\}}{\min\{m,n\}}\right)\frac{|X|}{m_0n_0}-\frac{4}{m_0n_0},
\end{align*}
and
\begin{align*}
    \frac{|S'|}{|R|}=\frac{|S|}{|R|}+\frac{|S'\backslash S|}{|R|}&\leq \frac{|S|}{|R|}+\frac{(2\lfloor\frac{m}{m_0+2}\rfloor+(m_0+2)) n+m(2\lfloor\frac{n}{n_0+2}\rfloor+(n_0+2))}{|R|}\\
    &\leq \frac{|X|}{(m_0+2)(n_0+2)}+\frac{2}{m_0+2}+\frac{m_0+2}{m}+\frac{2}{n_0+2}+\frac{n_0+2}{n}\\
    &\leq \frac{|X|}{m_0n_0}+\frac{4}{\min\{m_0,n_0\}}+\frac{2\max\{m_0+2,n_0+2\}}{\min\{m,n\}}.
\end{align*}
Given $\varepsilon>0$. Let $N_0$ be an index such that for every $t\geq N_0$,
\begin{align*}
    \sup_{s\geq t}\{D(s)\}<L_{\text{high}}+\varepsilon\ \ \ \text{ and }\ \ \ \inf_{s\geq t}\{d(s)\}>L_{\text{low}}-\varepsilon.
\end{align*}

Let $N_1\geq N_0$ be the index that we pick $X$ from $\mathcal{S}_k([m(N_1)]\times [n(N_1)])$ later. It means that $m_0=m(N_1)$ and $n_0=n(N_1)$. For the maximum density, we choose $N_1$ such that
\begin{align*}
    D(N_1)>L_{\text{high}}-\frac{1}{3}\varepsilon,\ \     \frac{m_0n_0}{(m_0+2)(n_0+2)}>1-\frac{1}{6 }\varepsilon,\ \text{ and }\ \frac{4}{m_0n_0}<\frac{1}{3}\varepsilon.
\end{align*}
Let $X$ be one reaching maximum density. Let $N_{\text{high}}\geq N_1$ be an index such that $\min\{m(t),n(t)\}\geq \frac{12}{\varepsilon}\max\{m_0,n_0\}$ for every $t\geq N_{\text{high}}$. Then, for $R=[m(t)]\times [n(t)]$ with $t\geq N_{\text{high}}$, we have
\begin{align*}
    L_{\text{high}}+\varepsilon >D(t)\geq\frac{|S'(X,R)|}{|R|}&> \left(1-\frac{1}{6}\varepsilon-\frac{1}{6}\varepsilon\right)\left(L_{\text{high}}-\frac{1}{3}\varepsilon\right)-\frac{1}{3}\varepsilon\geq L_{\text{high}}-\varepsilon.
\end{align*}
For the minimum density, we choose $N_1$ be such that
\begin{align*}
    d(N_1)<L_{\text{low}}+\frac{1}{2}\varepsilon\ \ \ \text{ and }\ \ \ \frac{4}{\min\{m_0,n_0\}}<\frac{1}{4}\varepsilon.
\end{align*}
Let $X$ be one reaching minimum density. Let $N_{\text{low}}\geq N_1$ be an index such that $\min\{m(t),n(t)\}\geq \frac{8}{\varepsilon}\max\{m_0+2,n_0+2\}$ for every $t\geq N_{\text{low}}$. Then, for $R=[m(t)]\times [n(t)]$ with $t\geq N_{\text{low}}$, we have
\begin{align*}
    L_{\text{low}}-\varepsilon<d(t)\leq \frac{|S'(X,R)|}{|R|}<\left(L_{\text{low}}+\frac{1}{2}\varepsilon\right)+\frac{1}{4}\varepsilon+\frac{1}{4}\varepsilon=L_{\text{low}}+\varepsilon.
\end{align*}
\end{proof}

\begin{proof}[Proof of Proposition \ref{limit_uniqueness}]
    Given two pairs $(m_0,n_0)$ and $(m_1,n_1)$ s.t. $m_i(t),n_i(t)\in \omega(1)$ for $i=0,1$. For $i=0,1$, by Proposition \ref{limit_existence}, the limit of the sequence $S_i=\{D(k,[m_i(t)]\times[n_i(t)])\}_{t=1}^\infty$ exists, we call it $L_i$. Consider the mixed pair $(m,n)$ defined by $$m(t):=m_r(q)\ \ \ \text{ and }\ \ \ n(t):=n_r(q),$$ where $t=2q+r\in\N$, $q\in\Z_{\geq 0}$, and $r\in \{0,1\}$. It is clear that $m(t),n(t)\in \omega(1)$. By Proposition \ref{limit_existence}, the limit of $S=\{D(k,[m(t)]\times[n(t)])\}_{t=1}^\infty$ exists, denoted by $L_{\text{high}}$. As $S_0$ and $S_1$ are subsequences of $S$, their limits agree with that of $S$, i.e., $L_0=L_{\text{high}}=L_1$. It suffices to prove the minimum density version by applying the same argument.
\end{proof}

In fact, these propositions could be generalized to higher dimension.
\begin{proposition}
    For $k,n\in\N$, $D(k,\Z^n)$ and $d(k,\Z^n)$ exist and are well-defined.
\end{proposition}

\section{Results for the Grid}\label{section3}

In this section, we study the grid $\mathbb{Z}^2$. We obtain nearly tight bounds on $D(k,\mathbb{Z}^2)$ (Theorem~\ref{thm2}), establish upper and lower bounds on $d(k,\mathbb{Z}^2)$ with gap $8k^{-1}$ (Theorem~\ref{thm3}), and determine the exact values of $D(3,\mathbb{Z}^2)$ and $d(3,\mathbb{Z}^2)$ (Theorem~\ref{thm1}).

Before proving these bounds, we show that the density spectrum of the grid forms an interval, so that it suffices to study the extremal densities.

\begin{theorem}
    For every $k\in\N_{\geq 2}$ and $\alpha\in(d(k,\Z^2),D(k,\Z^2))_{\R}$, there is a sequence $\{S_i\}_{i=1}^\infty$, where $S_i\in \mathcal{S}_k([\ell_i]^2)$ for every $i\in\N$, such that $\lim\limits_{i\to\infty} \frac{|S_i|}{\ell_i^2}=\alpha$.
\end{theorem}
\begin{proof}
    For simplicity, we write $D$ and $d$ for $D(k,\Z^2)$ and $d(k,\Z^2)$ in this proof. We identify $M_{\ell}(\{0,1\})$ with subsets of $[\ell]^2$ via
    \begin{align*}
        M \longleftrightarrow S=\{(i,j)\in[\ell]^2:M_{i,j}=1\}
    \end{align*}
    and define the density $\delta(M):=\frac{1}{\ell^2}|\{(i,j):M_{i,j}=1\}|$. Let $\overline{M}$ denote the matrix obtained from $M$ by adding a layer of zeros around it, i.e.,
    \begin{align*}
        \overline{M}:=\begin{pmatrix}
            0&0&0&0&0&0\\
            0&M_{1,1}&M_{1,2}&\cdots&M_{1,\ell}&0\\
            0&M_{2,1}&M_{2,2}&\cdots&M_{2,\ell}&0\\
            0&\vdots&\vdots&\ddots&\vdots&0\\
            0&M_{\ell,1}&M_{\ell,2}&\cdots&M_{\ell,\ell}&0\\
            0&0&0&0&0&0\\
        \end{pmatrix}\in M_{\ell+2}(\{0,1\}).
    \end{align*}
    For every $i\in\N_{\geq 2}$, let $S_{0,i},S_{1,i}\in\mathcal{S}_k([2^ik-2]^2)$ be saturated configurations reaching minimum and maximum density, respectively. Let $S_{j,i}'$ be obtained from $S_{j,i}$ by removing one point from each danger line (defined in the proof of Proposition \ref{limit_existence}), and regard $S_{0,i}'$ and $S_{1,i}'$ as matrices. Define two maps on $M_{2^ik}(\{0,1\})$ by
    \begin{align*}
        1(M):= \begin{pmatrix}
            M&\overline{S_{1,i}'}\\
            M&\overline{S_{1,i}'}
        \end{pmatrix},\ \ \ \  0(M):=\begin{pmatrix}
            M&\overline{S_{0,i}'}\\
            M&\overline{S_{0,i}'}
        \end{pmatrix}.
    \end{align*}
    Then for $j=0,1$, $$\delta(j(M))=\frac{1}{2}\delta(M)+\frac{1}{2}\delta(\overline{S_{j,i}}).$$
    Let $\alpha':=\frac{\alpha-d}{D-d}$ and write its binary expansion as $0.\alpha_1\alpha_2\alpha_3\cdots$.\footnote{It need not exclude infinite long consecutive $1$s in $\{\alpha_i\}_{i=1}^\infty$.} We construct matrices $\{M_i\}_{i=1}^\infty$ inductively. Fix an initial $M_0=\overline{S'}$ for some $S\in\mathcal{S}_k([2k-2]^2)$, where $S'$ is obtained from $S$ by removing one point from each danger line. For every $i\in\N$, we define $M_i$ by
    \begin{align*}
        M_i:=\alpha_1(\alpha_2(\cdots\alpha_{i}(M_{i-1})\cdots)).
    \end{align*}
    Each step ($\alpha_j(\cdot)$) doubles the side length, so the ground set for $M_i$ is $[ 2^{\binom{i+1}{2}+1}k]^2$. We set $\ell_i:= 2^{\binom{i+1}{2}+1}k$. Unfolding the recursion gives
    \begin{align*}
        \delta(M_i)=\frac{1}{2^i}\delta(M_{i-1})+\sum_{j=1}^i \frac{1}{2^j} \delta\left(\overline{S_{\alpha_j,\phi(i,j)}}\right),
    \end{align*}
    where $\phi(i,j):=\binom{i}{2}+1+(i-j)$. By Proposition \ref{limit_existence} and $\frac{(2^i-2)^2}{2^{2i}}\to1$, for any $\varepsilon>0$ there is an $N_0\in \N$ such that for all $i\geq N_0$,
    \begin{align*}
        \delta(\overline{S_{1,i}})>D-\varepsilon\ \ \text{ and }\ \ \delta(\overline{S_{0,i}})< d+\varepsilon.
    \end{align*}
    For $i$ large enough, with $2^{-i}<\varepsilon$ and $\phi(i,i)\geq N_0$, we have
    \begin{align*}
        \delta(M_i)=\frac{1}{2^i}\delta(M_{i-1})+\sum_{j=1}^i \frac{1}{2^j} \delta\left(\overline{S_{\alpha_j,\phi(i,j)}}\right)
        \in \sum_{j\in [i]:\alpha_j=1} \frac{D}{2^j}+\sum_{j\in [i]:\alpha_j=0} \frac{d}{2^j}\pm 2\varepsilon.
    \end{align*}
    By subtraction and division, we get
    \begin{align*}
        \sum_{j\in [i]:\alpha_j=1} \frac{D}{2^j}+\sum_{j\in [i]:\alpha_j=0} \frac{d}{2^j}= (D-d)\left(\sum_{j=1}^i \frac{\alpha_j}{2^j}\right)+d-\frac{d}{2^i}.
    \end{align*}
    Since $\sum\limits_{j=1}^i \frac{\alpha_j}{2^j}$ approximates $\alpha'$ with error at most $2^{-i}$, it follows that
    \begin{align*}
        \delta(M_i)\in (D-d)\alpha'+d\pm 4\varepsilon=\alpha\pm 4\varepsilon.
    \end{align*}
    Let $S_i'$ be the subset corresponding to $M_i$. Then $\frac{|S_i'|}{\ell_i^2}\to\alpha$ as $i\to\infty$. By the same argument as in the proof of Proposition \ref{limit_existence}, $S_i'$ avoids near $k$-lines since all danger lines are removed. Although $S_i'$ need not be saturated, we can extend it to a saturated configuration $S_i\supseteq S_i'$. The additional points $S_i\backslash S_i'$ must be contained in the empty boundary regions between blocks. A direct count shows that, as $i\to\infty$,
    \begin{align*}
        \frac{|S_i\backslash S_i'|}{\ell_i^2}\leq \frac{4\left(\binom{i+1}{2}+1\right)\ell_i}{\ell_i^2} \to 0,
    \end{align*}
    and hence $\frac{|S_i|}{\ell_i^2}\to \alpha$ as $i\to\infty$.
\end{proof}

\subsection{Nearly Tight Bounds on $D(k,\Z^2)$}

We split the proof into two parts: the upper bound and the lower bound. For the upper bound, we introduce the following lemma.

\begin{lemma}\label{lem:D:upper bound} For $k\in\N_{\geq 2}$, let $B$ be a finite $k$-board and $\mathcal{L}^\star$ be a family of pairwise-disjoint $k$-lines. Then $D(k,B)\leq 1-\frac{2}{k}+\frac{2|R|}{k|B|}$, where $R=R(\mathcal{L}^\star):=B\backslash (\cup_{L\in \mathcal{L}^\star} L)$.
\end{lemma}

By this lemma, for every $\ell\in\N$, as the disjoint horizontal $k$-lines family
\begin{align*}
    \mathcal{L}_{\ell}:=\{(x,y)+[0,k-1](1,0): x\in 1+k[0,\ell-1],\ y\in[k\ell]\}
\end{align*}
forms a partition of $[k\ell]^2$, we have
\begin{align*}
    D(k,[k\ell]^2)\leq 1-\frac{2}{k}+\frac{2|R|}{k|B|}=1-\frac{2}{k}
\end{align*}
since $R(\mathcal{L}_{\ell})=\emptyset$. Consequently, we obtain the upper bound
\begin{align*}
    D(k,\Z^2)=\lim_{\ell\to\infty} D(k,[k\ell]^2)\leq 1-\frac{2}{k}.
\end{align*}

\begin{proof}[Proof of Lemma \ref{lem:D:upper bound}]
    Let $S\in\mathcal{S}_k(B)$ be a saturated configuration reaching maximum density. By the definition of configuration, we have $|S\cap L|\leq k-2$ for every $k$-line $L$, which implies
    \begin{align*}
        |S|=\sum_{L\in\mathcal{L}^\star}|S\cap L|+|S\cap R|\leq (k-2)|\mathcal{L}^\star|+|R|=\left(1-\frac{2}{k}\right)|B|+\frac{2}{k}|R|.
    \end{align*}
    Dividing both sides by $|B|$, we get
    \begin{align*}
        D(k,B)=\frac{|S|}{|B|}\leq 1-\frac{2}{k}+\frac{2|R|}{k|B|}.
    \end{align*}
\end{proof}

The lower bound is based on Lemma \ref{lem:D:lower bound}, which provides required lower bound when $k$ is not a multiple of $3$. 
\begin{lemma}\label{lem:D:lower bound}
    For $k\in\N_{\geq 2}$ with $3\nmid k$, we have $D(k,\Z^2)\geq 1-\frac{2}{k}.$
\end{lemma}
For the case that $k$ is a multiple of $3$, noting that, in a hyper-rectangle, every near $k$-line contains a near $(k-1)$-line, hence near $(k-1)$-line avoidance implies near $k$-line avoidance. In particular, we have $\mathcal{C}_{k-1}([\ell]^2)\subseteq \mathcal{C}_k([\ell]^2)$, which implies
\begin{align*}
    D(k,[\ell]^2)\geq D(k-1,[\ell]^2).
\end{align*}
Combining Lemma \ref{lem:D:lower bound}, we derive $D(k,\Z^2)\geq D(k-1,\Z^2)\geq 1-\frac{2}{k-1}$.

To prove this lemma, we construct a configuration by removing zero set of some equations. Being an algebraic structure, a zero set is expected to have nearly the same number of points in any given line segments.
\begin{proof}[Proof of Lemma \ref{lem:D:lower bound}]\label{Proof of Theorem D:lower}
    Given $k\in\N_{\geq 2}$ with $3\nmid k$, we define
    \begin{align*}
        Z(r):=\{\ x=(x_1,x_2)\in \Z^2\ :\ x_1+2x_2=r\pmod{k}\}
    \end{align*}
    for each $r\in [0,k-1]$. First, we claim that the equality $|L\cap Z(0)|+|L\cap Z(1)|=2$ holds for every $k$-line on $\Z^2$. Fix a such $k$-line $L=v+[0,k-1]\Vec{d}$, where $v\in\Z^2$ and $\Vec{d}\in\{-1,0,1\}^2\backslash\{0\}$. Note that $v+a\Vec{d}\in Z(r)$ if and only if
    \begin{align*}
        (d_1+2d_2)a = r- (v_1+2v_2)\pmod{k}.
    \end{align*}
    If $d_1+2d_2$ and $k$ are coprime, there is a unique solution in $[0,k-1]$ to the equation. In this case, we have
    \begin{align*}
        |L\cap Z(0)|+|L\cap Z(1)|=1+1=2.
    \end{align*}
    As all possible values of $d_1+2d_2$ are $\pm1$, $\pm2$, and $\pm3$, the only non-coprime case is $d_1+2d_2=\pm2$ and $k\in 2\N$. In this case, the equation has zero or two solutions depending on the parity of $r-(v_1+2v_2)$. Since $0-(v_1+2v_2)$ and $1-(v_1+2v_2)$ must have different parity, it leads to $|L\cap Z(0)|+|L\cap Z(1)|=2$. Consider the subset $S_{\ell}:=[\ell]^2\backslash (Z(0)\sqcup Z(1))$ in the board $[\ell]^2$. By the claim, it is near $k$-line avoiding since it misses exactly two points from every $k$-line. Therefore, $S_{\ell}\in\mathcal{C}_k([\ell]^2)$. 
    For every $\ell\in \N$, the family $\mathcal{L}_{\ell}$ defined above is a partition of $[k\ell]^2$ consisting of $k$-lines. Hence
    \begin{align*}
        |S_{k\ell}|=\sum_{L\in\mathcal{L}_{\ell}} |S_{k\ell}\cap L|=\sum_{L\in\mathcal{L}_{\ell}}(k-2)=|\mathcal{L}_{\ell}|(k-2)=(k\ell)^2\left(1-\frac{2}{k}\right).
    \end{align*}
    Then we obtain $D(k,\Z^2)=\lim\limits_{\ell\to\infty} D(k,[k\ell]^2)\geq \lim\limits_{\ell\to\infty}\frac{|S_{k\ell}|}{(k\ell)^2}=1-\frac{2}{k}.$
\end{proof}

\subsection{Bounds on $d(k,\Z^2)$}

To lower bound $d(k,\Z^2)$, given a saturated configuration $S\in \mathcal{S}_k([\ell]^2)$, we count the number of pairs of points inside and outside of $S$ lying in some $k$-line. Then we obtain a relation between $|S|$ and $|[\ell]^2\backslash S|$; in particular, it implies:

\begin{lemma}
\label{lem:d:lower bound}
    For $k,\ell\in\N_{\geq 2}$, we have $d(k,[\ell]^2)\geq \frac{k-2}{16+(k-2)}=1-16k^{-1}-o(k^{-1}).$
\end{lemma}
\begin{proof}[Proof of Lemma \ref{lem:d:lower bound}]
    The bound is trivial for $k=2$, hence we assume $k\geq 3$. Let $S\in\mathcal{S}_k([\ell]^2)$ be a saturated configuration reaching minimum density. Define $N:=[\ell]^2\backslash S$. Consider the bipartite graph $G=(N\sqcup S,E)$ defined by
    \begin{align*}
        E:=\big\{\{x,y\}:x\in N,\ y\in S,\ \exists \text{ near $k$-line $L$ s.t. }y\in L\subseteq S\cup\{x\}\big\}.
    \end{align*}
    Given a $x\in N$, as $S\cup\{x\}\not\in\mathcal{C}_k([\ell]^2)$, there is a near $k$-line $L\subseteq S\cup\{x\}$. By the definition of $G$, we have $L\backslash\{x\}\subseteq N_G(x)$, hence $\delta(N)\geq k-2$. Conversely, given a $y\in S$ and a direction $\Vec{d}\in\{-1,0,1\}^2\backslash\{0\}$. For every $i\in\N$, let $n_i\in \N$ be the integer such that $x_i:=y+n_i\Vec{d}$ is the $i$-th empty position in the direction $\Vec{d}$ starting from $y$. Note that, for any $k$-line $L$ containing both $y$ and $x_i$, there are $i-1$ points in $L$, say $x_1,x_2,\cdots,x_{i-1}$, not included in $S\cup\{x_i\}$. Therefore, $x_i\not\in N_G(y)$ if $i\geq 3$. As there are eight directions, we have $\Delta(S)\leq 8\times 2=16$. Count the number of edges in $G$, we get
    \begin{align*}
        (k-2)|N|=\delta(N)|N|\leq |E(G)|\leq \Delta(S)|S|\leq 16|S|,
    \end{align*}
    which means $|N|\leq \frac{16}{k-2}|S|$. Then the minimum density is bounded below by  
    $$d(k,[\ell]^2)=\frac{|S|}{\ell^2}=\frac{|S|}{|N|+|S|}\geq \frac{|S|}{\frac{16}{k-2}|S|+|S|}=\frac{k-2}{16+(k-2)}.$$
\end{proof}

The upper bound is obtained via a specific construction. The key idea is that minimizing the number of points needed to control a given region yields a stronger upper bound. In particular, a filled $(k-2)\times (k-2)$ square almost controls the surrounding $(k+2)\times (k+2)$ square, in which each point controls at least four $k$-lines, making it a natural candidate for the construction. By tiling such pattern, we get:

\begin{lemma}
\label{lem:d:upper bound}
    For $k\in\N_{\geq 4}$, we have $d(k,\Z^2)\leq \frac{(k-2)^2+8}{(k+2)^2}=1-8k^{-1}+o(k^{-1}).$
\end{lemma}
\begin{proof}[Proof of Lemma \ref{lem:d:upper bound}]
    Consider the following subset inside $[k+2]^2$: The square $[3,k]^2$ in the center along with eight points surrounding each corners of the ground set
    $$\{(1,2), (2,1), (k+1,1), (k+2,2), (k+2,k+1), (k+1,k+2), (2,k+2), (1,k+1)\}.$$
    We call this subset $\mathcal{F}_k$, check the figure \ref{fig:filled-cog} for the example $k=5$.
    \begin{figure}[htp!]
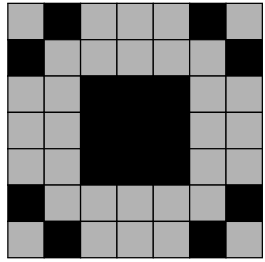

        \centering
        \ytableausetup{boxsize=1.1em}
        \ytableaushort[]{
{*(gray!60)}{*(black)}{*(gray!60)}{*(gray!60)}{*(gray!60)}{*(black)}{*(gray!60)},
{*(black)}{*(gray!60)}{*(gray!60)}{*(gray!60)}{*(gray!60)}{*(gray!60)}{*(black)},
{*(gray!60)}{*(gray!60)}{*(black)}{*(black)}{*(black)}{*(gray!60)}{*(gray!60)},
{*(gray!60)}{*(gray!60)}{*(black)}{*(black)}{*(black)}{*(gray!60)}{*(gray!60)},        
{*(gray!60)}{*(gray!60)}{*(black)}{*(black)}{*(black)}{*(gray!60)}{*(gray!60)},
{*(black)}{*(gray!60)}{*(gray!60)}{*(gray!60)}{*(gray!60)}{*(gray!60)}{*(black)},
{*(gray!60)}{*(black)}{*(gray!60)}{*(gray!60)}{*(gray!60)}{*(black)}{*(gray!60)}}
        \ytableausetup{boxsize=normal}
        \caption{$\mathcal{F}_5$}
        \label{fig:filled-cog}
    \end{figure}
    
    Note that, when $k\geq 4$, such pattern contains no near $k$-line and adding any remaining point in the ground set $[k+2]^2$ creates a near $k$-line within the ground set. For $\ell=n(k+2)\in(k+2)\N$, let $S$ be the subset built by tiling $\mathcal{F}_k$ over $[\ell]^2$. In particular, we define
    \begin{align*}
        S:=\bigcup_{v\in ((k+2)[0,n-1])^2} \mathcal{F}_k+v.
    \end{align*}
    
It is clear that $S$ contains no near $k$-lines. Given a point $x\in [\ell]^2\backslash S$, let $v_x\in ((k+2)[0,n-1])^2$ be such that $x\in [k+2]^2+v_x$. By the property of $\mathcal{F}_k$, we know there is a near $k$-line in $\mathcal{F}_k\cup \{x-v_x\}$. In other words, $S\cup\{x\}\supseteq (\mathcal{F}_k\cup\{x-v_x\})+v_x$ has a near $k$-line particularly inside $[k+2]^2+v_x$. Therefore, as $S$ is saturated, we have the upper bound
    \begin{align*}
        d(k,[\ell]^2)\leq \frac{|S|}{\ell^2}=\frac{\frac{\ell^2}{(k+2)^2}|\mathcal{F}_k|}{\ell^2}=\frac{(k-2)^2+8}{(k+2)^2}.
    \end{align*}
\end{proof}

Note that eight points around corners in $\mathcal{F}_k$ dominate when $k$ is small. In addition, these points control no points at all.\footnote{Although, they do control some points when $k=4$, but those points are also controlled by the central square from other tiles. Therefore, they are still redundant when minimizing the density.} A better bound for $k\leq 8$ could be obtained by removing these points. Consider $\mathcal{F}_k$ but removing these points from the ground set, denoted by $\mathcal{F}_k'$.
    \begin{figure}[htp!]
        \begin{minipage}{0.45\textwidth}
        \centering
        \ytableausetup{boxsize=1.1em}
        \ytableaushort[]{
{*(gray!60)}\none{*(gray!60)}{*(gray!60)}{*(gray!60)}\none{*(gray!60)},
\none{*(gray!60)}{*(gray!60)}{*(gray!60)}{*(gray!60)}{*(gray!60)}\none,
{*(gray!60)}{*(gray!60)}{*(black)}{*(black)}{*(black)}{*(gray!60)}{*(gray!60)},
{*(gray!60)}{*(gray!60)}{*(black)}{*(black)}{*(black)}{*(gray!60)}{*(gray!60)},        
{*(gray!60)}{*(gray!60)}{*(black)}{*(black)}{*(black)}{*(gray!60)}{*(gray!60)},
\none{*(gray!60)}{*(gray!60)}{*(gray!60)}{*(gray!60)}{*(gray!60)}\none,
{*(gray!60)}\none{*(gray!60)}{*(gray!60)}{*(gray!60)}\none{*(gray!60)}}
        \ytableausetup{boxsize=normal}
        \caption{$\mathcal{F}_5'$}
        \label{fig:cog}
        \end{minipage}
        \begin{minipage}{0.45\textwidth}
        \centering
        \ytableausetup{boxsize=0.9em}
        \ytableaushort[]{
        \none{*(red)}\none{*(red)}{*(red)}{*(red)}\none{*(red)}\none\none\none\none\none\none,
        \none\none{*(red)}{*(red)}{*(red)}{*(red)}{*(red)}{*(blue)}\none{*(blue)}{*(blue)}{*(blue)}\none{*(blue)},
        \none{*(red)}{*(red)}{*(black)}{*(black)}{*(black)}{*(red)}{*(red)}{*(blue)}{*(blue)}{*(blue)}{*(blue)}{*(blue)}\none,
        \none{*(red)}{*(red)}{*(black)}{*(black)}{*(black)}{*(red)}{*(red!50!blue)\textcolor{white}{\times}}{*(blue)}{*(black)}{*(black)}{*(black)}{*(blue)}{*(blue)},
        \none{*(red)}{*(red)}{*(black)}{*(black)}{*(black)}{*(red)}{*(red!50!blue)\textcolor{white}{\times}}{*(blue)}{*(black)}{*(black)}{*(black)}{*(blue)}{*(blue)},
        \none\none{*(red)}{*(red)}{*(red)}{*(red)}{*(red)}{*(blue)}{*(blue)}{*(black)}{*(black)}{*(black)}{*(blue)}{*(blue)},
        {*(green)}{*(red)}{*(green)}{*(green!50!red)\textcolor{white}{\times}}{*(green!50!red)\textcolor{white}{\times}}{*(red)}{*(green)}{*(red)}{*(blue)}{*(blue)}{*(blue)}{*(blue)}{*(blue)}\none,
        \none{*(green)}{*(green)}{*(green)}{*(green)}{*(green)}{*(yellow)}{*(blue)}{*(yellow)}{*(yellow!50!blue)\textcolor{white}{\times}}{*(yellow!50!blue)\textcolor{white}{\times}}{*(blue)}{*(yellow)}{*(blue)},
        {*(green)}{*(green)}{*(black)}{*(black)}{*(black)}{*(green)}{*(green)}{*(yellow)}{*(yellow)}{*(yellow)}{*(yellow)}{*(yellow)}\none\none,
        {*(green)}{*(green)}{*(black)}{*(black)}{*(black)}{*(green)}{*(green!50!yellow)\textcolor{black}{\times}}{*(yellow)}{*(black)}{*(black)}{*(black)}{*(yellow)}{*(yellow)}\none,
        {*(green)}{*(green)}{*(black)}{*(black)}{*(black)}{*(green)}{*(green!50!yellow)\textcolor{black}{\times}}{*(yellow)}{*(black)}{*(black)}{*(black)}{*(yellow)}{*(yellow)}\none,
        \none{*(green)}{*(green)}{*(green)}{*(green)}{*(green)}{*(yellow)}{*(yellow)}{*(black)}{*(black)}{*(black)}{*(yellow)}{*(yellow)}\none,
        {*(green)}\none{*(green)}{*(green)}{*(green)}\none{*(green)}{*(yellow)}{*(yellow)}{*(yellow)}{*(yellow)}{*(yellow)}\none\none,
        \none\none\none\none\none\none{*(yellow)}\none{*(yellow)}{*(yellow)}{*(yellow)}\none{*(yellow)}\none
        }
        \ytableausetup{boxsize=normal}
        \caption{the shifted covering}
        \label{fig:cog-covering}
        \end{minipage}
    \end{figure}
Consider the covering shown in Figure \ref{fig:cog-covering}. 

Let $S'$ be the union of $\mathcal{F}_k'$ in the covering and $C$ be the union of ground sets. It is clear that $S'$ is near $k$-line avoiding. Let $S''\in \mathcal{S}_k([\ell]^2)$ be the saturated configuration containing $S'$. As $S'$ is already saturated inside $C$, we know
\begin{align*}
    \frac{|S''\cap C|}{|C|}=\frac{|S'\cap C|}{|C|}=\frac{(k-2)^2}{k^2+2k+2}.
\end{align*}
Further, the difference $S''\backslash C$ lies in the band of width $k+1$ from the border, i.e., $|S'\backslash S|\leq 4(k+1)\ell$. Then we may conclude that
\begin{align*}
    d(k,[\ell]^2)\leq \frac{|S''|}{\ell^2}=\frac{|S''\cap C|}{\ell^2}+\frac{|S''\backslash C|}{\ell^2}\leq \frac{(k-2)^2}{k^2+2k+2}+\frac{4(k+1)}{\ell},
\end{align*}
which leads to the following lemma.

\begin{lemma}
\label{lem:d:upper bound(small)}
    For $k\in\N_{\geq 3}$, we have $d(k,\Z^2)\leq \frac{(k-2)^2}{k^2+2k+2}=1-6k^{-1}+o(k^{-1}).$
\end{lemma}

Notably, this construction works for the case $k=3$ but has a loose bound than Lemma \ref{lem:d:upper bound} as $k$ grows.

\subsection{Determining $D(3,\Z^2)$ and $d(3,\Z^2)$}

For $k=3$, according to Lemma \ref{lem:d:lower bound} and Lemma \ref{lem:d:upper bound(small)}, we have
\begin{align*}
    \frac{1}{17}=\frac{3-2}{16+(3-2)}\leq d(3,\Z^2)\leq \frac{(3-2)^2}{3^2+6+2}=\frac{1}{17},
\end{align*}
i.e., the minimum density is known. For the maximum density, we first lower bound it by a specific configuration, then we show that any saturated configuration reaching maximum density must be isomorphic to such configuration.

\begin{theorem}\label{D(3,Z^2)=1/5} $D(3,\Z^2)=\frac{1}{5}$.
\end{theorem}

\begin{proof}[Proof of Theorem \ref{D(3,Z^2)=1/5}]
Consider the subset
\begin{align*}
    Z:=\{\Vec{x}\in\Z^2\ :\ \gen{(2,-1),\Vec{x}}=1\pmod{5}\}.
\end{align*}
For each $\ell\in\N$, we define $S_\ell:=[\ell]^2\cap Z$. See Figure \ref{fig:S_10} for an illustration of $S_{10}$. Note that
$$\gen{(2,-1),\{-1,0,1\}^2\backslash\{0\}}=\{\pm1,\pm2,\pm3\}$$
is a set in which each element is coprime with $5$. Thus, following the proof of Lemma \ref{lem:D:lower bound}, every $5$-line intersects $S_\ell$ at exactly one point. As a result, we have $|S_\ell\cap L|\leq 1\leq 3-2$ for every $3$-line $L$; in other words, $S_\ell$ is near $3$-line avoiding.

    \begin{figure}[htp!]
        \begin{minipage}{0.51\textwidth}
        \centering
        \ytableausetup{boxsize=0.9em}
        \ytableaushort[]{
{*(white)}{*(white)}{*(black)}{*(white)}{*(white)}{*(white)}{*(white)}{*(black)}{*(white)}{*(white)},
{*(white)}{*(white)}{*(white)}{*(white)}{*(black)}{*(white)}{*(white)}{*(white)}{*(white)}{*(black)},
{*(white)}{*(black)}{*(white)}{*(white)}{*(white)}{*(white)}{*(black)}{*(white)}{*(white)}{*(white)},
{*(white)}{*(white)}{*(white)}{*(black)}{*(white)}{*(white)}{*(white)}{*(white)}{*(black)}{*(white)},
{*(black)}{*(white)}{*(white)}{*(white)}{*(white)}{*(black)}{*(white)}{*(white)}{*(white)}{*(white)},
{*(white)}{*(white)}{*(black)}{*(white)}{*(white)}{*(white)}{*(white)}{*(black)}{*(white)}{*(white)},
{*(white)}{*(white)}{*(white)}{*(white)}{*(black)}{*(white)}{*(white)}{*(white)}{*(white)}{*(black)},
{*(white)}{*(black)}{*(white)}{*(white)}{*(white)}{*(white)}{*(black)}{*(white)}{*(white)}{*(white)},
{*(white)}{*(white)}{*(white)}{*(black)}{*(white)}{*(white)}{*(white)}{*(white)}{*(black)}{*(white)},
{*(black)}{*(white)}{*(white)}{*(white)}{*(white)}{*(black)}{*(white)}{*(white)}{*(white)}{*(white)}
        }
        \ytableausetup{boxsize=normal}
        \caption{$S_{10}$}
        \label{fig:S_10}
        \end{minipage}
        \begin{minipage}{0.48\textwidth}
        \centering
        \ytableaushort[]{{*(gray!60)}{*(white)}{*(white)}{*(white)},{*(gray!60)}{*(black)}{*(white)\star}{*(white)},{*(gray!60)}{*(white)\star}{*(black)}{*(white)},{*(gray!60)}{*(gray!60)}{*(gray!60)}{*(gray!60)}}
        \caption{\centering Rare structure\\ $\star$: possible positions for $y$\\(Positions in gray are outside $[\ell]^2$.)}
        \label{fig:rare-structure-1}
        \end{minipage}
    \end{figure}

For $\ell\in 5\N$, $[\ell]^2$ can be partitioned into horizontal $5$-lines. Let $\mathcal{L}$ be a such partition, then we can lower bound the maximum density by
\begin{align*}
    D(3,[\ell]^2)\geq \frac{|S_{\ell}|}{\ell^2}=\frac{1}{\ell^2}\sum_{L\in\mathcal{L}} |S_{\ell}\cap L|=\frac{1}{\ell^2}\times \frac{\ell^2}{5}\times1=\frac{1}{5}.
\end{align*}
Let $\ell$ go to infinity, we derive $D(3,\Z^2)\geq \frac{1}{5}$.

We now upper bound the maximum density. Given a saturated configuration $S\in\mathcal{S}_3([\ell]^2)$. For every $x\in [2,\ell-1]^2\cap S$, its four adjacent points, $x+(\pm1,0)$ and $x+(0,\pm1)$, are not in $S$; otherwise, $S$ would not be near $3$-line avoiding. Conversely, for every $y\in [\ell]^2\backslash S$, if there are two of its adjacent points lying in $S$, it would become a near $3$-line except for the rare case shown in Figure \ref{fig:rare-structure-1}. Such case happens only if $$y\in \mathcal{R}:=\{(1,1), (2,2), (\ell,1), (\ell-1,2), (\ell,\ell), (\ell-1,\ell-1), (1,\ell), (2,\ell-1)\}.$$
Consider the bipartite graph $G=(S\sqcup N,E)$ where
\begin{align*}
    N:=[\ell]^2\backslash S,\ \text{ and }\ E:=\big\{ \{x,y\}\in S\times N\ :\ \normL{x-y}_2=1\big\}.
\end{align*}
By the discussion above, we have
\begin{align*}
    |E(G)|=\sum_{x\in S}d_G(x)\geq \sum_{x\in S\cap [2,\ell-1]^2}4= 4|S\cap [2,\ell-1]^2|\geq 4(|S|-4\ell).
\end{align*}
On the other hand, we know
\begin{align*}
    |E(G)|= \sum_{y\in N} d_G(y)\leq \sum_{y\in N\backslash\mathcal{R}}1+\sum_{y\in N\cap\mathcal{R}}4\leq |N|+3|\mathcal{R}|=\ell^2-|S|+24.
\end{align*}
Combining two inequalities and moving terms, we obtain $\ell^2+16\ell+24\geq 5|S|$. Let $S$ be one that reaching maximum density, then
\begin{align*}
    D(3,[\ell]^2)=\frac{|S|}{\ell^2}\leq \frac{\ell^2+16\ell+24}{5\ell^2}=\frac{1}{5}+o(1).
\end{align*}
Eventually, taking $\ell$ to infinity, we can see $D(3,\Z^2)\leq \frac{1}{5}$ and finish the proof.
\end{proof}

\begin{theorem}
    \label{thm:D:torus}
    For $k\in\N_{\geq 2}$, suppose we have $D(k,\Z^2)>1-\frac{2}{k}-\frac{1}{2k^2-k}$, then $D(k,(\Z/k\Z)^2)=1-\frac{2}{k}$.
\end{theorem}
\begin{proof}
    As $(\Z/k\Z)^2$ can be partitioned into $k$-lines, by lemma \ref{lem:D:upper bound}, we have the upper bound $D(k,(\Z/k\Z)^2)\leq 1-\frac{2}{k}$. Let $\{S_\ell\}_{\ell\in\N}$ be a sequence of saturated configurations such that, for every $\ell\in\N$, $S_\ell\in\mathcal{S}_k([\ell k]\times [\ell(2k-1)])$ and is of maximum density. Consider the partition
    \begin{align*}
        \bigcup_{v\in k[0,\ell-1]\times(2k-1)[0,\ell-1]} ([k]\times [2k-1])+v
    \end{align*}
    and define $S_{\ell,v}:=(S_{\ell}\cap (([k]\times [2k-1])+v))-v$ for every $v\in k[0,\ell-1]\times(2k-1)[0,\ell-1]$. By assumption, we have
    \begin{align*}
        \frac{1}{\ell^2}\sum_v \frac{|S_{\ell,v}|}{k(2k-1)}=\frac{|S_\ell|}{k(2k-1)\ell^2}=D(k,[\ell k]\times [\ell(2k-1)])>1-\frac{2}{k}-\frac{1}{2k^2-k}
    \end{align*}
    for sufficient large $\ell$. Hence, there is $v=v(\ell)$ such that $$\frac{|S_{\ell,v}|}{k(2k-1)}>1-\frac{2}{k}-\frac{1}{2k^2-k}.$$
    On the other hand, regarding $S_{\ell,v}$ as a (saturated) configuration in $[k]\times[2k-1]$, then, by Lemma \ref{lem:D:upper bound}, we know $$\frac{|S_{\ell,v}|}{k(2k-1)}\leq 1-\frac{2}{k}.$$ Noting $\frac{|S_{\ell,v|}}{k(2k-1)}\in\frac{\Z}{2k^2-k}$ and $(1-\frac{2}{k})(2k^2-k)\in \Z$, we may conclude that $\frac{|S_{\ell,v}|}{k(2k-1)}=1-\frac{2}{k}$. We now show that $S_{\ell,v}\cap [k]^2$ is a configuration of $(\Z/k\Z)^2$. Let $L$ be an arbitrary horizontal $k$-line lying in $[k]\times [2k-1]$. There is a partition $P=\{ L_i\}_{i=1}^{2k-1}$ of $[k]\times [2k-1]$ consisting of $k$-lines including $L$. As
    \begin{align*}
        1-\frac{2}{k}=\frac{|S_{\ell,v}|}{k(2k-1)}=\frac{1}{2k-1}\sum_{i=1}^{2k-1}\frac{|S_{\ell,v}\cap L_i|}{k}\leq \frac{1}{2k-1}\sum_{i=1}^{2k-1}\left(1-\frac{2}{k}\right)=1-\frac{2}{k},
    \end{align*}
    we must have $\frac{|S_{\ell,v}\cap L_i|}{k}=1-\frac{2}{k}$ for every $i\in [2k-1]$. In particular, $\frac{|S_{\ell,v}\cap L|}{k}=1-\frac{2}{k}$. Consider $L'=L+(0,1)$, if $L'\subseteq [k]\times [2k-1]$, we also have $\frac{|S_{\ell,v}\cap L'|}{k}=1-\frac{2}{k}$. As a result, we know that $x\in S_{\ell,v}\iff y\in S_{\ell,v}$, where $L\backslash L'=\{x\}$ and $L'\backslash L=\{y\}$. Therefore, for every $x\in [k]^2$, by taking the horizontal $k$-line starting at $x$ as $L$, we derive
    \begin{align}\label{shift}
        x\in S_{\ell,v}\iff x+(0,k)\in S_{\ell,v}.
    \end{align}
    Given a $k$-line $L$ of torus board $(\Z/k\Z)^2$, let $L_1$ be the $k$-subset in $[k]^2$ and $L_2$ be the $k$-line in $[k]\times [2k-1]$ such that they map to $L$ under the canonical projection $\Z^2\to(\Z/k\Z)^2$. Then we have
    \begin{align*}
        |(S_{\ell,v}\cap [k^2])\cap L|=|S_{\ell,v}\cap L_1|\overset{(\ref{shift})}{=}|S_{\ell,v}\cap L_2|=k-2.
    \end{align*}
    It implies that $S_{\ell,v}\cap [k]^2$ is near $k$-line avoiding on $(\Z/k\Z)^2$. Hence we obtain the lower bound $$D(k,(\Z/k\Z)^2)\geq \frac{|S_{\ell,v}\cap [k]^2|}{k^2}=1-\frac{2}{k}.$$
\end{proof}

\paragraph{}
Let $S\in\mathcal{S}_k((\Z/k\Z)^2)$ be a saturated configuration reaching maximum density. Regarding it as a subset of $[k]^2$, then $S+(k[0,n-1]\times k[0,n-1])$, the subset formed by tiling $S$ over $[nk]^2$, is a saturated configuration of $[nk]^2$. Then we have
\begin{align*}
    D(k,\Z^2)=\lim_{n\to\infty} D(k,[nk]^2)\geq\frac{|S|}{k^2}=D(k,(\Z/k\Z)^2).
\end{align*}
Therefore, Theorem \ref{thm:D:torus} implies the following:
\begin{corollary}
    For $k\in\N_{\geq 2}$, we have $D(k,\Z^2)\not\in (1-\frac{2}{k}-\frac{1}{2k^2-k},1-\frac{2}{k})_{\R}$.
\end{corollary}

\section{Results for Hypercubes}\label{section4}

In this section, we study hypercubes $[k]^d$ for $d\in\N$. We obtain bounds on $D(k,[k]^d)$ with gap $Ck^{-2}$, where $C$ is a constant depending only on $d$ (Theorem~\ref{thm4}), and derive exact values of $d(k,[k]^d)$ (Theorem~\ref{thm5}).

\subsection{Bounds on $D(k,[k]^d)$}
The upper bound is simple. As $[k]^d$ can be partitioned into a number of $k$-lines, by Lemma \ref{lem:D:upper bound}, we have $D(k,[k]^d)\leq 1-\frac{2}{k}$. The lower bound follows the spirit of proof of Lemma \ref{lem:D:lower bound}. However, instead of removing two zero sets from the board, we first partition the board into hypercubes, then remove a zero set from each part.

\begin{lemma}\label{lem:k^d:D:lower bound}
    For every $d\in\N$, there are constants $C=C(d)>0$ and $N=N(d)\in\N$ such that
        \begin{align*}
            D(k,[k]^d)\geq 1-\frac{2}{k}-\frac{C}{k^2} 
        \end{align*}
        holds for $k\geq N$.
\end{lemma}
\begin{proof}[Proof of Lemma \ref{lem:k^d:D:lower bound}]
    Let $\ell=\ell(k,d)$ be the largest integer not greater than $\frac{k}{2}$ such that it is coprime with $i$ for every $i\in [d]$. Clearly, the difference btween $\frac{k}{2}$ and $\ell$ is not greater than $d!$. Like the \hyperref[Proof of Theorem D:lower]{proof of Lemma \ref{lem:D:lower bound}}, for every $r\in\Z$, we define
    \begin{align*}
        Z(r):=\{\ \Vec{x}\in \Z^d\ :\ \gen{(1,1,\cdots,1),\Vec{x}}=r\pmod{\ell}\}.
    \end{align*}
    We claim that, for every $\ell$-line $L=x+[0,\ell-1]\Vec{d}$ on $\Z^d$ with $\Vec{d}\in\{0,1\}^d\backslash\{0\}$, we have $|L\cap Z(1)|=1$.\label{claim} Recall that, as in the \hyperref[Proof of Theorem D:lower]{proof of Lemma \ref{lem:D:lower bound}}, we have
    \begin{align*}
        |L\cap Z(r)|=1\ \ \text{ if and only if}\ \  \gcd(\gen{(1,1,\cdots,1),\Vec{d}},\ell)=1.
    \end{align*}
    Note that, for $\Vec{d}\in\{0,1\}^t\backslash\{0\}$, we have $\gen{(1,1,\cdots,1), \Vec{d}}=\normL{\Vec{d}}_1\in [d]$. Then, as $\ell$ is coprime with $[d]$, the claim is proved.
    Set $B=B(\ell,d):=[0,\ell-1]^d\backslash Z(1)$. In the following, we will use projections and reflections of $B$ as the building blocks to form a configuration $S^\star=S^\star(k,d)\in\mathcal{C}_k([k]^d)$ such that
    \begin{align*}
        \frac{|S^\star|}{k^d}\geq 1-\frac{2}{k}-\frac{4d!}{(k-2d!)k}.
    \end{align*}
    With such configuration, let $C(d):=8d!$ and $N(d):=4d!$, then we derive
    \begin{align*}
        D(k,[k]^d)\geq \frac{|S^\star|}{k^d}\geq 1-\frac{2}{k}-\frac{4d!}{(k-2d!)k}\geq 1-\frac{2}{k}-\frac{C}{k^2}
    \end{align*}
    for $k\geq N$, closing the proof. \qedhere
    \subsubsection*{Building up the configuration $S^\star$}
    We start with defining the core $X:=[\ell,k-\ell+1]^d$, a hypercube centered in $[k]^d$. Also, we define its \emph{interior} $\text{int}X:=[\ell+1,k-\ell]^d$ and \emph{boundary} $\partial X:=X\backslash \text{int}X$. For each point $x$ in the core, we paste a modified building block $\phi(x)$ on it. The map $\phi$ is defined by
    \begin{align*}
        \phi:X&\to \left\{2^{\Z^d}\to 2^{\Z^d}\right\}\\
        x&\mapsto \left[Y\mapsto R_{\text{rev}(x)}(Y)\cap E_{\text{int}(x)}\right],
    \end{align*}
    where $\text{rev}(x):=x^{-1}(\ell)\subseteq [d]$ and $\text{int}(x):=x^{-1}([\ell+1,k-\ell])\subseteq [d]$ are the set of reversing and interior indices respectively; $R_I$ negates the $i$-th coordinate for each $i\in I$, and $E_J$ is the subspace of $\Z^d$ with the $j$-th coordinate zero for each $j\in J$. Define
    \begin{align*}
        S^\star:=\bigsqcup\limits_{x\in X} (x+\phi(x)(B))\subseteq [k]^d.
    \end{align*}
    Figure \ref{fig:example_6-hypercube-D-construction} displays $S^\star$ and several types of building blocks for $k=8$ and $d=2$.%
    
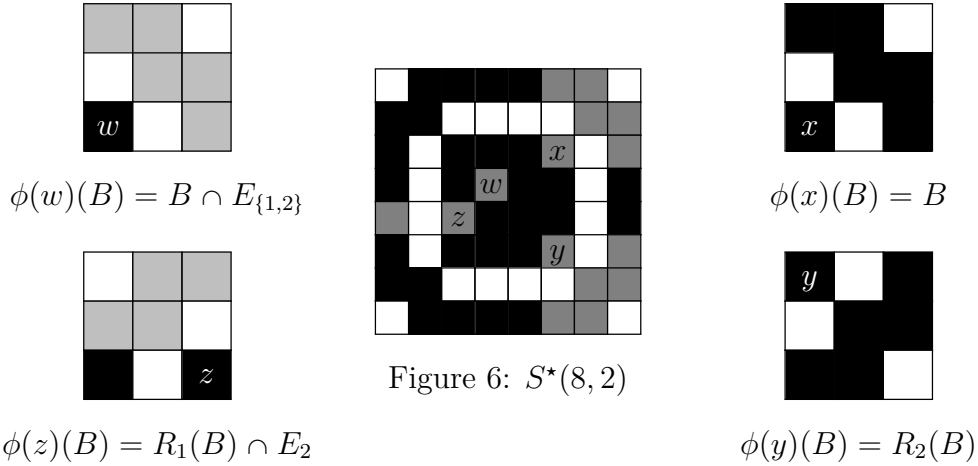
\begin{center}
        \begin{figure}[htp!]
        \centering
        \begin{minipage}{0.3\textwidth}
        
        \centering
        \ytableaushort[]{
{*(lightgray)}{*(lightgray)}{*(white)},
{*(white)}{*(lightgray)}{*(lightgray)},
{*(black)\textcolor{white}{w}}{*(white)}{*(lightgray)}
        }
        \caption*{$\phi(w)(B)=B\cap E_{\{1,2\}}$}
        
        \centering
        \vspace{1em}
        \ytableaushort[]{
{*(white)}{*(lightgray)}{*(lightgray)},
{*(lightgray)}{*(lightgray)}{*(white)},
{*(black)}{*(white)}{*(black)\textcolor{white}{z}}
        }
        \caption*{$\phi(z)(B)=R_1(B)\cap E_2$}
        \end{minipage}
        \begin{minipage}{0.3\textwidth}
        \centering
        \ytableausetup{boxsize=1em}
        \ytableaushort[]{
{*(white)}{*(black)}{*(black)}{*(black)}{*(black)}{*(gray)}{*(gray)}{*(white)},
{*(black)}{*(black)}{*(white)}{*(white)}{*(white)}{*(white)}{*(gray)}{*(gray)},
{*(black)}{*(white)}{*(black)}{*(black)}{*(black)}{*(gray)x}{*(white)}{*(gray)},
{*(black)}{*(white)}{*(black)}{*(gray)w}{*(black)}{*(black)}{*(white)}{*(black)},
{*(gray)}{*(white)}{*(gray)z}{*(black)}{*(black)}{*(black)}{*(white)}{*(black)},
{*(black)}{*(white)}{*(black)}{*(black)}{*(black)}{*(gray)y}{*(white)}{*(gray)},
{*(black)}{*(black)}{*(white)}{*(white)}{*(white)}{*(white)}{*(gray)}{*(gray)},
{*(white)}{*(black)}{*(black)}{*(black)}{*(black)}{*(gray)}{*(gray)}{*(white)}
        }
        \ytableausetup{boxsize=normal}
        \caption{$S^\star(8,2)$}
        \label{fig:example_6-hypercube-D-construction}
        \end{minipage}
        \begin{minipage}{0.3\textwidth}
        \centering
        \ytableaushort[]{
{*(black)}{*(black)}{*(white)},
{*(white)}{*(black)}{*(black)},
{*(black)\textcolor{white}{x}}{*(white)}{*(black)}
        }
        \caption*{$\phi(x)(B)=B$}
        \vspace{1em}
        \ytableaushort[]{
{*(black)\textcolor{white}{y}}{*(white)}{*(black)},
{*(white)}{*(black)}{*(black)},
{*(black)}{*(black)}{*(white)}
        }
        \caption*{$\phi(y)(B)=R_2(B)$}
        \end{minipage}
    \end{figure}
    \end{center}
    In this example, for $x$, we paste the original $B$ on it; for $y$, the block but negating the $2$-th coordinate is pasted; for $z$ and $w$, in order to fit in the space $[k]^d$, the blocks are truncated on $E_2$ and $E_{\{1,2\}}$, respectively. 
    
    Before we verify that $S^\star$ is a configuration, we first lower bound the density of this subset. Given $x\in X$, if $x\in \text{int}X$, $\phi(x)(B)$ is a single point and $|\text{int}(x)|=d$. If $x\in\partial X$, the ground space $\phi(x)([0,\ell-1]^d)$ can be partitioned into $\ell^{d-1-|\text{int}(x)|}$ $\ell$-lines of direction in $\{0,1\}^d\backslash\{0\}$. By \hyperref[claim]{the claim above}, we know $|L\cap \phi(x)(B)|=\ell-1$ for every $L$ in the partition, hence we obtain $|\phi(x)(B)|= (\ell-1)\ell^{d-1-|\text{int}(x)|}$. In either case, we have
    \begin{align}
        d(x):=\frac{|\phi(x)(B)|}{|\phi(x)([0,\ell-1]^d)|}\geq \frac{(\ell-1)\ell^{d-1-|\text{int}(x)|}}{\ell^{d-|\text{int}(x)|}}=1-\frac{1}{\ell}.\label{ineq:lowerbound of buildingblock}
    \end{align}
    As the ground spaces $\{\phi(x)([0,\ell-1]^d)\}_{x\in X}$ form a partition of $[k]^d$, $\frac{|S^{\star}|}{k^d}$ can be regarded as the weighted-average of $\{d(x)\}_{x\in X}$ by setting the weight of $x$ to be $|\phi(x)([0,\ell-1]^d)|$. Then (\ref{ineq:lowerbound of buildingblock}) implies that $\frac{|S^{\star}|}{k^d}\geq 1-\frac{1}{\ell}$. Hence, by $\ell>\frac{k}{2}-d!$, we obtain the desired lower bound
    \begin{align*}
        \frac{|S^\star|}{k^d}\geq 1-\frac{1}{\ell}=1-\frac{2\ell}{\ell k}-\frac{k-2\ell}{\ell k}\geq 1-\frac{2}{k}-\frac{4d!}{(k-2d!)k}.
    \end{align*}
    \subsubsection*{Near $k$-line Avoidance}
    \begin{claim}\label{fully_through}
        For every $k$-line $L$ in $[k]^d$, there are two points $x,y\in X$ such that $$|x+\phi(x)([0,\ell-1]^d)\cap L|=|y+\phi(y)([0,\ell-1]^d)\cap L|=\ell.$$
        In particular, $R_{\text{rev}(x)}^{-1}(L-x)$ and $R_{\text{rev}(y)}^{-1}(L-y)$ are $\ell$-lines with direction in $\{0,1\}^d\backslash\{0\}$.
    \end{claim}
    With this claim, it is sufficient to show the near $k$-line forbiddenness of $S^{\star}$. Given a $k$-line $L$, let $x,y$ be such points in Claim \ref{fully_through}. Following \hyperref[claim]{the previous claim}, we can see that $S^{\star}$ misses two points on $L$, one from $x+\phi(x)([0,\ell-1]^d\backslash B)$ and the other from $y+\phi(y)([0,\ell-1]^d\backslash B)$. As a consequence, $S^{\star}$ contains no near $k$-lines, i.e., $S^{\star}\in\mathcal{C}_k([k]^d)$. \qedhere
    \begin{proof}[Proof of Claim \ref{fully_through}]
    Say $L=v+[0,k-1]\Vec{d}$, where $\Vec{d\in\{-1,0,1\}^d\backslash \{0\}}$. Define $x,y\in [k]^d$ by setting
    \begin{align*}
        x_i:=\begin{cases}
            \ell&\text{ if $d_i=1$ or ($d_i=0$ and $v_i\leq\ell$)},\\
            k-\ell+1&\text{ if $d_i=-1$ or ($d_i=0$ and $v_i\geq k-\ell+1$)},\\
            v_i&\text{ otherwise},
        \end{cases}
    \end{align*}
    and
    \begin{align*}
        y_i:=\begin{cases}
            k-\ell+1&\text{ if $d_i=1$ or ($d_i=0$ and $v_i\leq\ell$)},\\
            \ell&\text{ if $d_i=-1$ or ($d_i=0$ and $v_i\geq k-\ell+1$)},\\
            v_i&\text{ otherwise},
        \end{cases}
    \end{align*}
    for every $i\in [d]$. We now show that the $v+[0,\ell-1]\Vec{d}$ is contained in $x+\phi(x)([0,\ell-1]^d)$. We show the containment holds coordinatewisely. Recall that $\phi(x)([0,\ell-1]^d)$ is a hyper-rectangle, hence there are $a_i<b_i\in \Z$ such that $\phi(x)([0,\ell-1]^d)=\prod_{i=1}^d [a_i,b_i]$. Denote $v_t:= v+t\Vec{d}$. for each $t\in [0,\ell-1]$. Consider three cases:
    \begin{enumerate}
        \item[$(d_i=1)$] In this case, we have $x_i=\ell$ and hence $\phi(x)$ negates the $i$-th coordinate. Thus $$[a_i,b_i]=\ell-[0,\ell-1]=[1,\ell].$$

        Noting that $(v_t)_i=t+1$ in this case, we have $(v_t)_i\in [a_i,b_i]$ for each $t\in [0,\ell-1]$.
        \item[$(d_i=-1)$] Similar to the first case, as $x_i=k-\ell+1$, $\phi(x)$ does nothing on the $i$-th coordinate. Hence
        $$[a_i,b_i]=(k-\ell+1)+[0,\ell-1]=[k-\ell+1,k].$$
        Observing that $(v_t)_i=k-t$, then we get $(v_t)_i\in [a_i,b_i]$ for each $t\in[0,\ell-1]$.
        
        \item[$(d_i=0)$] By the discussion above, we already have
        \begin{align*}
            [a_i,b_i]=\begin{cases}
                [1,\ell]&\text{ if }\ v_i\leq \ell,\\
                [k-\ell+1,k]&\text{ if }\ v_i\geq k-\ell   +1.
            \end{cases}
        \end{align*}
        For the case $\ell<v_i<k-\ell+1$, $\phi(x)$ eliminates all but points with the $i$-th coordinate zero. Therefore
        \begin{align*}
            [a_i,b_i]=v_i+[0,0]=[v_i,v_i].
        \end{align*}
        In either case, $v_i\in [a_i,b_i]$. Noting that $(v_t)_i=v_i$ for all $t$, then the result follows.
    \end{enumerate}
    Due to the definition of $x$, we have $R_{\text{rev}(x)}^{-1}(\Vec{d})=R_{\text{rev}(x)}(\Vec{d})\subseteq \{-1,0\}^d$. Then, by changing the starting point of the $R_{\text{rev}(x)}^{-1}(v+[0,\ell-1]\Vec{d}-x)$, it can be regarded as a $\ell$-line with direction in $\{0,1\}^d$. One can show that $v+[k-\ell+1,k]\Vec{d}$ is contained in $y+\phi(y)([0,\ell-1]^d)$ and it is a desired $\ell$-line in a similar way.
    \end{proof}
\end{proof}

\subsection{Determining $d(k,[k]^d)$}\label{d(k,[k]^d)}
In this section, we will provide the exact value of $d(k,[k]^d)$. For the upper bound, we consider a smaller hypercube $S:=[2,k-1]^d$. Clearly, $S$ contains no near $k$-line since it excludes two endpoints from every $k$-line. It remains to show that $S$ is saturated. Given a point $x\in [k]^d\backslash S$. Consider the $k$-line $L:=x+[0,k-1]\Vec{d}_x$, where the direction $\Vec{d}$ is defined by
    \begin{align*}
        \Vec{d}_x:= \sum_{i: x_i=1} e_i-\sum_{i:x_i=k}e_i.
    \end{align*}
    We have $(S\cup\{x\})\cap L=x+[0,k-2]\Vec{d}_x$, which means $S\cup\{x\}$ contains a near $k$-line. Therefore
    \begin{align*}
        d(k,[k]^d)\leq \frac{|S|}{k^d}=\frac{(k-2)^d}{k^d}=\left(1-\frac{2}{k}\right)^d.
    \end{align*}

On the other hand, let $S_0\in\mathcal{S}_k([k]^d)$ be a saturated configuration reaching minimum density. Define $X:= S\backslash S_0$ and $Y:=S_0\backslash S$. Obviously, we have $|Y|\leq |X|$. We now claim there is an injection $f:X\to Y$. If so, then we have $|X|\leq |Y|$ and hence $|S_0|=|S|=(k-2)^d$, which finishes the proof. For every $x\in X$, by the fact that $x\not\in S_0$ and $S_0$ is saturated, there is a $k$-line $L=L(x)$ such that $|L\backslash S_0|=\{x,y\}$. Consider two cases:
\begin{enumerate}[(i)]
    \item Suppose $y\not\in S$, there is only one point $z$ in $L\cap Y=(L\backslash S)\cap S_0$. Set $f(x):=z$.
    \item Suppose $y\in S$, there are two points $z$ and $w$ in $L\cap Y$. WLOG, let $z$ be the one that $y$ is not between $x$ and $z$. Then we set $f(x):=z$.
\end{enumerate}

To verify the injectivity, given $x_1,x_2\in X$ with $x_1\neq x_2$. If $L(x_1)=L(x_2)$, let $z$ and $w$ be that in (ii) when considering $x_1$, then we have 
$$f(x_1)=z\neq w=f(x_2).$$
If $L(x_1)\neq L(x_2)$, following Lemma \ref{unique line in [k]^d}, we have $(L_1\backslash S)\cap(L_2\backslash S)=\emptyset$. Then we have $f(x_1)\neq f(x_2)$ due to the facts $f(x_i)\in L(x_i)\backslash S$ for $i=1,2$. \qed

\begin{lemma}\label{unique line in [k]^d}
    Given $x\in [k]^d\backslash S$, the $k$-line $L=x+[0,k-1]\Vec{d}_x$ is the only $k$-line such that $x\in L$ and $L\cap S\neq \emptyset$.
\end{lemma}
\begin{proof}
    Fix $x\in [k]^d\backslash S$. Let $L'=y+[0,k-1]\Vec{d}$ be a $k$-line in $[k]^d$ such that $x=y+a\Vec{d}$. Suppose $x$ is not an endpoint of $L'$, i.e., $a\neq 0$ and $a\neq k-1$. Since $x\in [k]^d\backslash S$, there is an index $i_0\in [d]$ such that $x_{i_0}\in \{1,k\}$. We note that $d_{i_0}=0$. Otherwise, one of $$\{y_{i_0},y_{i_0}+a,y_{i_0}+(k-1)\}\subseteq \{1,k\}\ \ \ \text{ or }\ \ \ \{y_{i_0},y_{i_0}-a,y_{i_0}-(k-1)\}\subseteq \{1,k\}$$
    would hold, which is impossible. Hence $z_{i_0}=x_{i_0}\not\in [2,k-1]$ for every $z\in L'$, and thus $L'\cap S=\emptyset$.
    
    It follows that if $L'\cap S\neq\emptyset$, then $x$ must be an endpoint of $L'$, and we may write $L'=x+[0,k-1]\Vec{d'}$. We now determine $\Vec{d'}$ coordinate-wisely. For each $i\in [d]$, consider three cases:
    \begin{enumerate}
        \item If $x_i=1$, then $d'_i\neq 1$ implies either $L'\cap S=\emptyset$ or $x+\Vec{d'}\not\in [k]^d$. Hence $d'_i=1$.
        \item If $x_i=k$, a symmetric argument shows that $d'_i=-1$.
        \item If $x_i\in [2,k-1]$, then $d'_i\neq 0$ would imply $(x+(k-1)\Vec{d'})_i\not\in [k]$, which contradicts that $L'\subseteq [k]^d$. Hence $d_i=0$.
    \end{enumerate}
    Therefore, we obtain $$\Vec{d'}=\sum_{i:x_i=1}e_i-\sum_{i:x_i=k}e_i=\Vec{d}_x,$$ and hence $L'=L$. 
\end{proof}

\section{Conjectures and Extensions}\label{section5}

Before turning to conjectures, we summarize the current results. The best known bounds and exact values are listed below:

\begin{center}
{\renewcommand{\arraystretch}{1.9}
\setlength{\tabcolsep}{4pt}
\small
\scalebox{0.96}{
\begin{tabular}{@{}|C||C|C|C||C|C|C|@{}}\hline
    B&\multicolumn{3}{C||}{d(k,B)}&\multicolumn{3}{C|}{D(k,B)}\\\hline\hline
    \multirow{2}{*}{$\Z^2$}&k=3&4\leq k\leq 8&9\leq k&k=3&3\not\mid k& 3\mid k\\\cline{2-7}
    &\frac{1}{17}&\frac{k-2}{k+14}\leq\cdot\leq \frac{(k-2)^2}{k^2+2k+2}&1-\frac{16}{k}\lesssim\cdot\lesssim 1-\frac{8}{k}&\frac{1}{5}&1-\frac{2}{k}&1-\frac{2}{k-1}\leq\cdot\leq 1-\frac{2}{k}\\\hline\hline
    [k]^d&\multicolumn{3}{C||}{(1-\frac{2}{k})^d}&\multicolumn{3}{C|}{1-\frac{2}{k}-\frac{C(d)}{k^2}\leq\cdot\leq 1-\frac{2}{k}}\\\hline
\end{tabular}
}
}
\end{center}

In addition, we have that $D(k,\Z^2)<1-\frac{2}{k}$ implies $D(k,\Z^2) \leq 1-\frac{2}{k}-\frac{1}{2k^2-k}$.

\subsection{Conjectures}

\subsubsection{The Grid}
From Lemma \ref{lem:D:upper bound} and the proof of Theorem \ref{thm:D:torus}, the inequalities
\begin{align*}
    D(k,(\Z/k\Z)^2)\leq D(k,\Z^2)\leq 1-\frac{2}{k}
\end{align*}
hold for every $k\in\Z_{\geq 2}$. Note that, when $k$ is not a multiple of $3$, the configuration $S_k$ in the proof of Lemma \ref{lem:D:lower bound} avoids not only $k$-lines in $\Z^2$ but also $k$-lines in $(\Z/k\Z)^2$ by the construction, which leads to
\begin{align*}
    1-\frac{2}{k}=\frac{|S_k|}{k^2}\leq D(k,(\Z/k\Z)^2)\leq D(k,\Z^2)\leq 1-\frac{2}{k}
\end{align*}
and forces the middle two terms, $D(k,(\Z/k\Z)^2)$ and $D(k,\Z^2)$, to be $1-\frac{2}{k}$. On the other hand, it is easy to see that $D(3,(\Z/3\Z)^2)=\frac{1}{9}$ since every two points in $(\Z/3\Z)^2$ have a common ($3$-)line. So both inequalities are strict when $k=3$. Further, by exhaustive searching, we have $D(6,(\Z/6\Z)^2)<1-\frac{2}{6}$, which means at least one of these inequalities is strict. It motivates the following conjectures.

\begin{conjecture}\label{conj:D(k,Z2)} For every $k\in 3\N$, $D(k,\Z^2)<1-\frac{2}{k}$.
\end{conjecture}
\begin{conjecture}\label{conj:D(k,Z/kZ)} For every $k\in 3\N$, $D(k,(\Z/k\Z)^2)<D(k,\Z^2)$.
\end{conjecture}

Suppose the Conjecture \ref{conj:D(k,Z2)} is true, then the inequality
\begin{align*}
    D(k,\Z^2)\geq D(k,(\Z/k\Z)^2)
\end{align*}
forces the bound $D(k,(\Z/k\Z)^2)< 1-\frac{2}{k}$ to hold for $k\in 3\N$. Conversely, according to Theorem \ref{thm:D:torus}, if $D(k,(\Z/k\Z)^2)<1-\frac{2}{k}$ holds for some $k$, then we must have
\begin{align*}
    D(k,\Z^2)\leq 1-\frac{2}{k}-\frac{1}{2k^2-k}<1-\frac{2}{k}
\end{align*}
for the same $k$. Hence, for each $k\in3\N$, the following are equivalent. 
\begin{enumerate}
    \item[$\bullet$] $D(k,\Z^2)<1-\frac{2}{k}$.
    \item[$\bullet$] $D(k,(\Z/k\Z)^2)<1-\frac{2}{k}$.
    \item[$\bullet$] There is no subset $A\subseteq (\Z/k\Z)^2$ satisfying $|L\cap A|=2$ for every axis-parallel or diagonal line $L$ in $(\Z/k\Z)^2$.
\end{enumerate}

The third statement characterizes when a configuration has density $1-\frac{2}{k}$, where the subset $A$ should be interpreted as its complement. Since this formulation avoids explicit reference to configurations, we restate Conjecture \ref{conj:D(k,Z2)} accordingly.

\begin{conjecture}{(Restatement of Conjecture \ref{conj:D(k,Z2)})}
    For every $k\in 3\N$, there is no subset $A\subseteq (\Z/k\Z)^2$ satisfying $|L\cap A|=2$ for every axis-parallel or diagonal line $L$.
\end{conjecture}

As Theorem \ref{thm2} provides a lower bound $1-\frac{2}{k-1}$ on $D(k,\Z^2)$. If we can show that $D(k,(\Z/k\Z)^2)<1-\frac{2}{k-1}$ for every $k\in3\N\backslash \{3\}$, then Conjecture \ref{conj:D(k,Z/kZ)} holds. Recall that $D(k,(\Z/k\Z)^2)\in\frac{\N}{k^2}$, we have 
\begin{align*}
    D(k,(\Z/k\Z)^2)<1-\frac{2}{k-1} \iff D(k,(\Z/k\Z)^2)< \frac{\lceil k^2(1-\frac{2}{k-1})\rceil}{k^2}.
\end{align*}
In other words, one way to prove Conjecture \ref{conj:D(k,Z/kZ)} is to check that if every saturated configuration in $\mathcal{S}_k((\Z/k\Z)^2)$ has size lesser than $\lceil k^2(1-\frac{2}{k-1})\rceil=k^2-2k-2$. Equivalently, we need to check that, for every saturated configuration, the size of its complement is greater than $2k+2$ or not. Then we give a conjecture.

\begin{conjecture}\label{conj:ult}
    For every $k\in 3\N$, there is no subset $A\subseteq (\Z/k\Z)^2$ with $|A|\leq 2k+2$ satisfying $|L\cap A|\geq 2$ for every axis-parallel or diagonal line $L$. 
\end{conjecture}

By the discussion above, we know Conjecture \ref{conj:ult} implies Conjecture \ref{conj:D(k,Z/kZ)}. Moreover, it also implies Conjecture \ref{conj:D(k,Z2)} due to the restatement.

For the minimum density, recall that $1-16k^{-1}\lesssim d(k,\Z^2)\lesssim 1-8k^{-1}$ as $k\to\infty$. The problem is to characterize such constant.
\begin{question} Determine the constant $c\in [8,16]_{\R}$ such that $d(k,\Z^2)\sim1-ck^{-1}$.
\end{question}
At present, the values near $16$ are less possible. Indeed, if the coefficient were close to $16$, then the average degrees of the two parts of the graph $G$ defined in Lemma \ref{lem:d:lower bound} would have to be close to $k-2$ and $16$, respectively. This would require that, in $G$, the average degree on one side is nearly minimized while that on the other side is nearly maximized, which suggests that such a scenario cannot occur.

\subsubsection{Hypercubes}

For the maximum density, the configuration $S^{\star}$ constructed in the proof of Lemma \ref{lem:k^d:D:lower bound} has density
\begin{align*}
    \frac{|S^{\star}|}{k^d}=1-\frac{2}{k}
\end{align*}
if $\frac{k}{2}\in\N$ and is coprime with $d!$. It suggests the following conjecture.
\begin{conjecture}\label{conj:[k]^d}
    For every $k\in\N_{\geq 2}$ and $d\in\N$, $D(k,[k]^d)=1-\frac{2}{k}$.
\end{conjecture}
Consider the subset $$X_k:=[k]^2\backslash\{(x,y)\in [k]^2: x+y=1\text{ or }0\pmod{k}\}.$$
See Figure \ref{fig:X_6} for illusions of $X_5$ and $X_6$. It is obvious that $X_k\in\mathcal{C}_k([k]^2)$. Hence $$D(k,[k]^2)\geq \frac{|X_k|}{k^2}=\frac{k^2-2k}{k^2}=1-\frac{2}{k},$$
which means Conjecture \ref{conj:[k]^d} is true for $d=2$.
\begin{figure}[htp!]
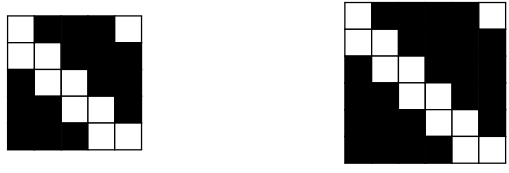

    \centering
    \ytableausetup{boxsize=0.8em}
    \begin{minipage}{0.3\textwidth}
        \centering
    \ytableaushort[]{
    {*(white)}{*(black)}{*(black)}{*(black)}{*(white)},
    {*(white)}{*(white)}{*(black)}{*(black)}{*(black)},
    {*(black)}{*(white)}{*(white)}{*(black)}{*(black)},
    {*(black)}{*(black)}{*(white)}{*(white)}{*(black)},
    {*(black)}{*(black)}{*(black)}{*(white)}{*(white)}
    }
    \end{minipage}
    \begin{minipage}{0.3\textwidth}
        \centering
    \ytableaushort[]{
    {*(white)}{*(black)}{*(black)}{*(black)}{*(black)}{*(white)},
    {*(white)}{*(white)}{*(black)}{*(black)}{*(black)}{*(black)},
    {*(black)}{*(white)}{*(white)}{*(black)}{*(black)}{*(black)},
    {*(black)}{*(black)}{*(white)}{*(white)}{*(black)}{*(black)},
    {*(black)}{*(black)}{*(black)}{*(white)}{*(white)}{*(black)},
    {*(black)}{*(black)}{*(black)}{*(black)}{*(white)}{*(white)}
    }
    \end{minipage}
    \ytableausetup{boxsize=normal}
    \caption{$X_5$ and $X_6$}
    \label{fig:X_6}
\end{figure}

\newpage
For the minimum density, recall that $d(k,[k]^d)=(1-\frac{2}{k})^d$. In fact, there is a family of configurations reaching minimum density, and this family includes the previous example $[2,k-1]^d$. For each $i\in [\lfloor\frac{k}{2}\rfloor]$, define $$W_i:=([k]\backslash \{i,(k+1)-i\})^d.$$
Given a $k$-line $L=v+[0,k-1]d$ in $[k]^d$, the two points, $v+(i-1)d$ and $v+(k-i)d$, are not contained in $W_i$. Hence $W_i$ misses at least two points from every $k$-line, and therefore is near $k$-line avoiding. On the other hand, given a point $x\in [k]^d\backslash W_i$, consider the $k$-line $$L=v+[0,k-1]\Vec{d},$$ where $v:=x-(i-1)\Vec{d}$ and $\Vec{d}:=\sum_{j:x_j=i}e_j-\sum_{j:x_j=(k+1-i)}e_j$. It is clear that $x=v+(i-1)\Vec{d}$ and $\overline{x}=v+(k-i)\Vec{d}$ are the only points in $L$ not contained in $W_i$. Hence we have $|L\backslash W_i|=2$, and therefore $W_i$ is saturated.

It appears that these are the only configurations reaching minimum density when $d\geq 2$. Indeed, this is true for $k=3$ and all $d\geq 2$, and also for $k=4$ and $d=2,3$. This motivates the following question.

\begin{question}
    For every $d\in\N_{\geq 2}$, are $\{W_i\}_{i\in [\lfloor\frac{k}{2}\rfloor]}$ the only saturated configurations, in $\mathcal{S}_k([k]^d)$, reaching minimum density?
\end{question}

\subsection{Extensions}

\paragraph{Winner Determination} For a fixed board $B$, the monochromatic $k$-in-a-row game is a finite perfect-information game, and hence the outcome is determined: it is either a first-player win (\textbf{FP}), a second-player win (\textbf{SP}), or a draw. In this setting, draws occur if and only if the family of $k$-lines is empty, which is a degenerate case that we typically exclude, so the game is either \textbf{FP} or \textbf{SP}.

Unlike many positional games, the strategy-stealing argument does not apply here, as the additional move available to the first player may be disadvantageous; in particular, the first player cannot always guarantee a win. A natural problem is to determine the winning player for a given board $B$, or more ambitiously, to describe an explicit winning strategy.

For instance, the outcome on $[k]^d$ is completely characterized: it is \textbf{FP} when $k$ is odd and \textbf{SP} when $k$ is even. This result comes from a pairing strategy for the winning player. See pages 46-48 in \cite{ku2025monochromatic} for the construction. 

\paragraph{Biased Game} One natural extension is the \emph{biased $(p:q)$ game}, in which the first and second players claim $p$ and $q$ points per move, respectively. The cases $p\geq k$ or $q\geq k$ are trivial, so we restrict to $\max\{p,q\}<k$. 

A heuristic suggests that the player with the larger bias should have an advantage, and this is indeed the case when the board is sufficiently large. Let Alice and Bob denote the players with biases $\alpha$ and $\beta$, respectively, and suppose $\alpha > \beta$. Consider the quantity $N$, defined as the minimum number of points needed to complete a $k$-line from the current state. Alice's strategy is based on the value of $N$: If $N\leq \alpha$, Alice wins without doubt. If not, Alice claims $\alpha$ points such that $N$ is decreased by $1$. (This can be done once the board size is large enough.) In this setting, Bob cannot win; otherwise, one would obtain a contradiction $\beta \geq N-1 \geq \alpha$ before the very last move. Consequently, the most interesting case is the balanced biased $(p:p)$ game with $1 < p < k$, where the advantage of larger bias disappears.

\paragraph{Multiplayer Games} For $t\in\N_{\geq 3}$, we can consider the $t$-player variant: There are $t$ players. Each player, in turns, claims a position from the board until a $k$-line is created. It comes a natural question: Under this circumstance, which player will win the game? Although it is similar to the $2$-player game, but there are some cases that, at some moment, a player cannot win but their move can determine the winner. For instance, consider the $5$-player monochromatic $3$-in a row played on $[3]^2$: If player $1$ starts with taking $(2,2)$, then player $3$ wins. If player $1$ starts with taking other point, only player $3$, $4$, and $5$ have chance to win. In either case, player $1$ cannot win.
\begin{figure}[htp!]
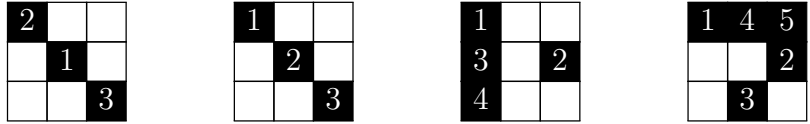

    \centering
    \ytableausetup{boxsize=1.2em}
    \begin{minipage}{0.2\textwidth}
        \centering
        \ytableaushort[]{
        {*(black)\textcolor{white}{2}}{}{},{}{*(black)\textcolor{white}{1}}{},{}{}{*(black)\textcolor{white}{3}}
        }
    \end{minipage}%
    \begin{minipage}{0.2\textwidth}
        \centering
        \ytableaushort[]{
        {*(black)\textcolor{white}{1}}{}{},{}{*(black)\textcolor{white}{2}}{},{}{}{*(black)\textcolor{white}{3}}
        }
    \end{minipage}%
    \begin{minipage}{0.2\textwidth}
        \centering
        \ytableaushort[]{
        {*(black)\textcolor{white}{1}}{}{},{*(black)\textcolor{white}{3}}{}{*(black)\textcolor{white}{2}},{*(black)\textcolor{white}{4}}{}{}
        }
    \end{minipage}%
    \begin{minipage}{0.2\textwidth}
        \centering
        \ytableaushort[]{
        {*(black)\textcolor{white}{1}}{*(black)\textcolor{white}{4}}{*(black)\textcolor{white}{5}},{}{}{*(black)\textcolor{white}{2}},{}{*(black)\textcolor{white}{3}}{}
        }
    \end{minipage}
    \caption{several configurations}
    \label{multiplayer}
    \ytableausetup{boxsize=normal}
\end{figure}\ \\
Thus, further analysis requires additional restrictions or assumptions on the players.

\paragraph{$s$-Saturated Configurations} We can extract the concept of saturated configuration and quantize it. Given a $k$-board $B$. For every $s\in [k]$, we define the \emph{$s$-configurations}
\begin{align*}
    \mathcal{C}_{k,s}(B):=\{S\subseteq B: |L\backslash S|\geq s\ \forall \text{ $k$-line $L$}\}.
\end{align*}
Then we accordingly define the \emph{saturated $s$-configurations}, $\mathcal{S}_{k,s}(B)$. Further, if $B$ is finite, we also define \emph{$s$-density spectrum}, $\mathcal{D}_{k,s}(B)$, and \emph{extremal $s$-densities}, $D_s(k,B)$ and $d_s(k,B)$. Notably, the original problem is corresponding to the case $s=2$.

Several results can be extended to this setting. For instance, by the same argument as in the proof of Lemma \ref{lem:D:upper bound}, we can obtain
\begin{align*}
    D_s(k,B)\leq 1-\frac{s}{k}+\frac{s|R|}{k|B|}
\end{align*}
for every $s\in[k]$, where $R$ is the complement of union of some disjoint $k$-lines. And, by adapting the proof of Lemma \ref{lem:d:lower bound}, we have
\begin{align*}
    d_s(k,B)\geq \frac{k-s}{sN(B)+(k-s)},
\end{align*}
where $N(B):=\max_{x\in B}|\{\text{$k$-lines containing $x$}\}|$. See pages 54-58 in \cite{ku2025monochromatic} for detailed proofs. However, bounds obtained from specific constructions still require further discussion and modification.

\paragraph{Extremal Board Problem}

The original monochromatic $k$ in a row problem asks how large or small the density of a saturated configuration on a given board can be.A related question is the following: for a fixed board size $n$, which boards maximize or minimize the extremal densities? To formalize this question, we restrict attention to $n$-subsets of $\Z^2$. For every $B\in\binom{\Z^2}{n}$, we regard $B$ as a board by defining $$\mathcal{L}(B):=\{L:L\text{ is a $k$-line on $\Z^2$},\ L\subseteq B\}.$$
Then the question is to determine the terms
\begin{align*}
    D_{\min}(k,n):=\min_{B\in\binom{\Z^2}{n}} D(k,B)\ \ \ \ \text{ and }\ \ \ \ d_{\max}(k,n):=\max_{B\in\binom{\Z^2}{n}} d(k,B).
\end{align*}
However, without additional assumptions determining $d_{\max}(k,n)$ becomes trivial. Indeed, if a point $x\in B$ is not contained in any $k$-line of $B$, then $x$ must belong to every saturated configuration. In fact, for every $n\in\N$, we can find a $n$-subset containing no $k$-line, it implies $d_{\max}(k,n)=1$. Therefore, it is natural to consider subsets that every point is contained in some $k$-lines.

\bibliography{references}

\end{document}